\documentclass[11pt]{smfart}
\usepackage{color}
\usepackage{amssymb,verbatim}
\usepackage{amsthm,array,amssymb,amscd,amsfonts,amssymb,latexsym, url}
\usepackage{amsmath}
\usepackage[all]{xy}
\usepackage[french]{babel}

\newcounter{spec}
{\end{list}}

\renewcommand{\P}{{\mathbf P}}

\newcommand{\Z}{{\mathbb Z}}
\newcommand{\Q}{{\mathbb Q}}
\newcommand{\C}{{\mathbb C}}

\renewcommand{\L}{{\mathbb L}}
\newcommand{\oi}{\hskip1mm {\buildrel \simeq \over \rightarrow} \hskip1mm}

\newcommand{\Br}{{\operatorname{Br}}}
\newcommand{\Ker}{{\operatorname{Ker}}}
\newcommand{\Div}{{\operatorname{Div}}}

\newcommand{\Hom}{{\operatorname{Hom}}}

\newcommand{\by}[1]{\overset{#1}{\longrightarrow}}
\newcommand{\iso}{\by{\sim}}
\renewcommand{\lim}{\varprojlim}

\numberwithin{equation}{section}

\newfont{\gothic}{eufb10}

\renewcommand{\qed}{{\hfill$\square$}}

\newtheorem{theo}{Th\'{e}or\`{e}me}[section]
\newtheorem{prop}[theo]{Proposition}

\newtheorem{lem}[theo]{Lemme}
\newtheorem{cor}[theo]{Corollaire}
\theoremstyle{definition}
\newtheorem{defi}[theo]{D\'efinition}
\theoremstyle{remark}
\newtheorem{rema}[theo]{Remarque}

\newcommand{\bthe}{\begin{theo}}
\newcommand{\ble}{\begin{lem}}
\newcommand{\bpr}{\begin{prop}}
\newcommand{\bco}{\begin{cor}}
\newcommand{\bde}{\begin{defi}}
\newcommand{\ethe}{\end{theo}}
\newcommand{\ele}{\end{lem}}
\newcommand{\epr}{\end{prop}}
\newcommand{\eco}{\end{cor}}
\newcommand{\ede}{\end{defi}}

\newcommand{\Pic}{\operatorname{Pic}}

\def\ov{\overline}
\def\Gal{{\rm Gal}}

\def\Z{{\bf Z}}
\def\Q{{\bf Q}}
\def\R{{\bf R}}
\def\C{{\bf C}}
\def\k{{\overline k}}
\DeclareFontFamily{U}{wncy}{}
\DeclareFontShape{U}{wncy}{m}{n}{%
<5>wncyr5%
<6>wncyr6%
<7>wncyr7%
<8>wncyr8%
<9>wncyr9%
<10>wncyr10%
<11>wncyr10%
<12>wncyr6%
<14>wncyr7%
<17>wncyr8%
<20>wncyr10%
<25>wncyr10}{}
\DeclareMathAlphabet{\cyr}{U}{wncy}{m}{n}

\begin{document}

\title[Brauer non ramifi\'e]{Groupe de Brauer non ramifi\'e de quotients par un groupe fini}
\author{J.-L. Colliot-Th\'el\`ene}
 \address{C.N.R.S., Universit\'e Paris Sud\\Math\'ematiques, B\^atiment 425\\91405 Orsay Cedex\\France}
 \email{jlct@math.u-psud.fr}
\date{9 janvier 2012}
\maketitle

 \begin{abstract}
 Soit $k$ un corps, $G$ un groupe fini, $G \hookrightarrow SL_{n,k}$ un plongement.
 Pour $k$ alg\' ebriquement clos, Bogomolov a donn\' e une formule pour le groupe
 de Brauer non ramifi\'e du quotient $SL_{n,k}/G$.  On examine ce que donne sa m\'ethode
 sur un corps $k$ quelconque (de caract\'eristique nulle). Par cette m\'ethode purement alg\'ebrique,
 on retrouve et \'etend  des r\'esultats obtenus  par Harari et par Demarche 
 au moyen de m\' ethodes arithm\'etiques, comme la trivialit\'e du groupe de Brauer non ramifi\'e
  pour $k=\Q$
 et $G$ d'ordre impair.
  \end{abstract}

 \begin{altabstract} Let $k$ be a field, $G$ a finite group, $G \hookrightarrow SL_{n,k}$ an embedding.
 For $k$ an algebraically closed field, Bogomolov gave a formula for the unramified Brauer group of the 
 quotient $SL_{n,k}/G$. We develop his method over any  characteristic zero field.
 This purely algebraic method enables us to recover and generalize results of Harari and of Demarche
 over number fields, such as the triviality of the unramified Brauer group for $k=\Q$
 and $G$ of odd order.
  \end{altabstract}

  \section{Introduction} 

 Soient $k$ un corps de caract\'eristique  z\'ero, $\ov{k}$ une cl\^oture alg\'ebrique de $k$,
et $\frak{g}=\Gal(\ov{k}/k)$.
\`A toute $k$-vari\'et\'e alg\'ebrique g\'eom\'etriquement int\`egre $W$,
de corps des fonctions $k(W)$,  on associe d'une part son 
groupe de Brauer cohomologique $\Br(W)$, d'autre part le groupe
de Brauer non ramifi\'e $\Br_{nr}(k(W)/k)$ qu'on notera
ici simplement $\Br_{nr}(k(W))$. 
Pour la d\'efinition et les propri\'et\'es de base de ces groupes,
on renvoie \`a \cite{grot},   \cite{ctsb} et \cite{cts}. Rappelons simplement que si $W$ est   lisse sur $k$, alors la restriction au point g\'en\'erique induit
 une inclusion 
$$ \Br(W) \hookrightarrow  \Br(k(W))$$
et que si $W$ est projective et lisse, alors cette inclusion induit un isomorphisme
$$ \Br(W) \iso  \Br_{nr}(k(W)).$$

 Soit   $n\geq 2$ un entier, $X=SL_{n,k}$ et $G$ un groupe fini vu comme $k$-groupe 
 constant, et $G  \subset SL_{n,k}$ un $k$-plongement.
Notons $Y$ l'espace homog\`ene quotient $X/G$. Le corps des fonctions $k(Y)$ de $Y$
est le corps $k(X)^G$ des invariants du corps des fonctions $k(X)$ sous l'action de~$G$.

Ces notations seront conserv\'ees dans tout l'article.

On se propose d'examiner le groupe de Brauer non ramifi\'e $\Br_{nr}(k(Y))$.
Ce groupe 
(\`a isomorphisme non unique pr\`es)
 ne d\'epend que de $G$, il ne d\'epend
ni de $n$ ni du plongement
$G \subset SL_{n,k}$.

Comme $Y$ est une $k$-vari\'et\'e lisse int\`egre, on a 
$\Br_{nr}(k(Y)) \subset \Br(Y)$.

On va suivre ici la   m\'ethode purement alg\'ebrique
qu'avait utilis\'ee  Bogomolov  (\cite{bogomolov}, \cite[Thm. 7.1]{cts})
lorsque $k$ est alg\'ebriquement clos. 

Ceci  permet   d'\'etendre
certains des r\'esultats \'etablis 
 par des m\'ethodes arith\-m\'etiques 
par  Harari \cite{harari}
et par Demarche \cite{demarche} sur le sous-groupe \og{alg\'ebrique}\fg
$$\Br_{nr,1}(k(Y)) : = \Ker [\Br_{nr}(k(Y)) \to \Br(\overline{k}(Y))],$$
  sous-groupe  form\'e des classes 
qui s'annulent
par passage au corps des fonctions $\overline{k}(Y)$ de $\overline{Y}=Y\times_{k}\overline{k}$.

  En particulier, on montre ici (Corollaire \ref{racines}) que pour tout corps $k$
  de caract\'eristique z\'ero ne poss\'edant qu'un nombre fini de racines de l'unit\'e,
   pour  $Y=SL_{n,k}/G$ avec $G$ fini constant d'ordre
  premier au cardinal du groupe des racines de l'unit\'e dans~$k$,
  on a  $\Br(k)=\Br_{nr}(k(Y))$.
  
  Etant donn\'e un corps $k$ de caract\'eristique z\'ero, pour tout premier $p$ impair et tout entier $r\geq 1$,
 l'extension $k(\mu_{p^r})/k$ est cyclique. Il en est encore ainsi pour $p=2$ si $-1$
 est un carr\'e dans $k$.
 
 \medskip
 
  On note $Cyc(G,k)$ la  condition : 
  
  {\it Pour tout sous-groupe cyclique $\Z/2^r \subset G$, 
   l'extension $k(\mu_{2^r})/k$ est cyclique.}

\medskip

Pour $A$ un groupe ab\'elien, $n>0$ un entier et $l$ un nombre premier, on note
$A[n] \subset A$ le sous-groupe annul\'e par $n$ et on note $A\{l\} \subset A$ le sous-groupe
de torsion $l$-primaire.

 \section{Les r\'esultats}

On  note  $\Br_{nr}^0(k(Y)) \subset \Br(Y)$ le groupe de Brauer  normalis\'e,
c'est-\`a-dire le groupe des \'el\'ements qui s'annulent 
  en l'image du point $1\in SL_{n}(k)$.

On a la suite exacte de restriction-inflation
$$0 \to H^2(G,k(X)^{\times})  \to \Br(k(Y)) \to \Br(k(X)).$$

\begin{lem}\label{lemmeinitial}
Soit  $H^2_{nr}(G,k(X)^{\times}) \subset H^2(G,k(X)^{\times}) \subset \Br(k(Y))$
le sous-groupe    form\'e des \'el\'ements  non ramifi\'es.

(i) Dans $\Br(k(X)^G)$, on  a $H^2_{nr}(G,k(X)^{\times}) \cap \Br(k) =0$.

(ii) Le groupe $H^2_{nr}(G,k(X)^{\times}) \subset  H^2(G,k(X)^{\times}) \subset \Br(k(Y))$
 co\"{\i}ncide avec le groupe $\Br_{nr}^0(k(Y))$. 
 
 (iii) Le groupe $H^2_{nr}(G,k(X)^{\times}) = \Br_{nr}^0(k(Y))$ est fini.
\end{lem}
\begin{proof}
On a des inclusions naturelles compatibles  $\Br(k) \hookrightarrow \Br_{nr}(k(Y))$
 et $\Br(k) \hookrightarrow \Br_{nr}(k(X))$, et cette derni\`ere fl\`eche est un isomorphisme
car $X$ est une vari\'et\'e $k$-rationnelle.   Ceci \'etablit (i) et (ii). L'\'enonc\'e de finitude (iii) vaut plus
g\'en\'eralement 
 pour les corps de fonctions de vari\'et\'es projectives, lisses, g\'eom\'etriquement connexes
qui sont g\'eom\'etriquement unirationnelles.  Rappelons bri\`evement la d\'emonstration.
Pour toute $k$-vari\'et\'e projective, lisse, g\'eom\'etriquement int\`egre $X$, 
notant $\ov{X}=X\times_k\ov{k}$, on a une suite exacte
$$ \Br(k) \to \Ker [\Br(X) \to \Br(\ov{X})] \to H^1(\frak{g},\Pic(\ov{X})). $$
Si $X$ est g\'eom\'etriquement unirationnelle, alors le module galoisien  $\Pic(\ov{X})$
est un groupe ab\'elien libre de type fini, et le groupe $ \Br(\ov{X})$ est fini. Ceci suffit 
\`a conclure.
  \end{proof}

\medskip
 
Notons   que  dans la d\'ecomposition d'un \'el\'ement non ramifi\'e de $$H^2(G,k(X)^{\times})\subset \Br(k(Y))$$
en ses composantes $p$-primaires pour chaque premier $p$, chacune de celles-ci est non ramifi\'ee.

 \medskip 

  L'\'enonc\'e suivant est une variante sur un corps non alg\'ebriquement clos d'un th\'eor\`eme
 de Fischer 
 (cf. \cite[Prop. 4.3]{cts}).

\begin{prop}\label{Fischer}
Supposons le groupe $G$ ab\'elien et  la condition $Cyc(G,k)$ satisfaite.
 Il existe une $k$-vari\'et\'e
$Z$ telle que le produit  $Y \times_{k} Z$ est $k$-birationnel \`a un espace projectif sur $k$.
\end{prop}
\begin{proof}  On dit que deux $k$-vari\'et\'es $W_{1}$ et $W_{2}$
sont stablement $k$-birationnellement \'equivalentes s'il existe
des espaces projectifs $\P^r_{k}$ et $\P^s_{k}$ tels que $W_{1}\times_{k}\P^r_{k}$
est $k$-birationnel \`a $W_{2 }\times_{k}\P^s_{k}$.  D'apr\`es une version
du  \og lemme sans nom \fg,
pour un groupe fini $G$, des $k$-groupes sp\'eciaux $X_{1}$ et $X_{2}$
 et des plongements $G \subset X_{1}$
et $G \subset X_{2}$, le quotient $X_{1}/G$ est stablement $k$-birationnel
\`a $X_{2}/G$ (\cite[Prop. 4.9]{cts}). Les produits de groupes $SL_{n,k}$
($n$ variable)  sont sp\'eciaux. 
Si donc l'on a un plongement $G_{1} \subset X_{1}$, un plongement $G_{2} \subset X_{2}$
et un plongement $G_{1} \times G_{2} \subset X$, o\`u chacun des groupes $X$ et $X_{i}$
est un groupe sp\'ecial lin\'eaire, alors $X/G$ est stablement $k$-birationnel \`a
$(X_{1} \times X_{2})/(G_{1} \times G_{2}) = X_{1}/G_{1} \times X_{2}/G_{2}$.

Pour \'etablir la proposition, on est donc ramen\'e au cas o\`u 
$G=\Z/p^m$ est un groupe ab\'elien cyclique $p$-primaire, pour  $p$ un nombre premier.
L'extension $k(\mu_{p^m})/k$ est cyclique si $p$ est impair, et c'est encore le cas
si $p=2$ d'apr\`es l'hypoth\`ese de la proposition.
 
 Suivant  Voskresenski\u{\i}, rappelons comment cela se traduit en termes
 de tores.
 Soit $g={\rm Gal}(k(\mu_{p^m})/k)$ et soit $\hat{G}=\Hom(G,\mu_{p^m})=\mu_{p^m}$.
 Le groupe ab\'elien libre de base les \'el\'ements de $\mu_{p^m}$, qui est un $g$-module de
 permutation,
 s'envoie de fa\c con \'evidente et $g$-\'equivariante sur le groupe $\mu_{p^m}$.
 On note $\hat{T}$ le noyau de cette application. On a donc la suite exacte
 de $g$-modules continus discrets :
$$0 \to \hat{T} \to \oplus_{\zeta \in \mu_{p^m}} \Z.\zeta \to \hat{G} \to 0.$$
Par dualit\'e, on obtient une suite exacte de $k$-groupes alg\'ebriques de type multiplicatif
$$ 1 \to G \to P \to T \to 1,$$
o\`u $G=\Z/p^m$, $T$ est un $k$-tore et   $P$ est un $k$-tore quasi-trivial,
en particulier un $k$-groupe sp\'ecial. 
Ainsi (\cite[Prop. 4.9]{cts}), pour tout $k$-plongement $G \subset X=SL_{n,k}$, le quotient
$X/G$ est 
 stablement $k$-birationnel au $k$-tore $T$. 

 Le $k$-tore $T$ est d\'eploy\'e par une extension cyclique. 
 Rappelons comment ceci implique   qu'il existe un $k$-tore $T_{1}$ tel que
 $T \times_{k}T_{1}$ est $k$-birationnel \`a un espace projectif, ce qui ach\`evera la
 d\'emonstration.
 
On sait   (Voskresenski\u{\i}, Endo-Miyata)  que
 tout $k$-tore $T$  d\'eploy\'e par une extension galoisienne $K/k$
 admet une r\'esolution flasque, c'est-\`a-dire qu'il existe une suite 
 de $k$-tores  d\'eploy\'es par l'extension $K/k$
 $$1 \to F \to P \to T \to 1,$$
o\`u  $P$ est un $k$-tore quasi-trivial, donc $k$-birationnel \`a un espace projectif
 et $F$ est un $k$-tore flasque (voir \cite[Lemme 3 p.~18]{RET} et \cite[\S 5]{RET}). 
 Supposons $K/k$ cyclique.
  D'apr\`es un th\'eor\`eme de
 Endo et Miyata (voir \cite[Prop. 2, p.~184]{RET}) tout $k$-tore flasque $S$ d\'eploy\'e par
 une telle extension  est un facteur direct d'un $k$-tore quasitrivial : il existe
 un $k$-tore $S'$ tel que le $k$-tore $S \times_{k}S'$ soit $k$-isomorphe 
 \`a un $k$-tore quasitrivial.
 Le th\'eor\`eme 90 de Hilbert et la r\'esolution flasque ci-dessus assurent
 alors
que la $k$-vari\'et\'e $k$-rationnelle $P$ est $k$-birationnelle au produit
 $T \times_{k} F$.  
 \end{proof}

Cette proposition implique imm\'ediatement le th\'eor\`eme :

\begin{theo}\label{theoprincip1}

(i)    Soit $G$ ab\'elien. Si   la condition $Cyc(G,k)$ est satisfaite, alors
$\Br_{nr}^0(k(X)^G)= 0$.

(ii)  Soit $G$ ab\'elien. On a $2.\Br_{nr}^0(k(X)^G)= 0$.

Soit $G$ fini quelconque.

 (iii) Pour tout sous-groupe ab\'elien $A \subset G$,
 l'image de  $H^2_{nr}(G,k(X)^{\times})$
  dans $H^2(A,k(X)^{\times})$ est annul\'ee par 2.

(iv)  
Si la condition $Cyc(G,k)$ est satisfaite,
 pour  tout sous-groupe ab\'elien $A \subset G$,
 l'image de  $H^2_{nr}(G,k(X)^{\times})$
  dans $H^2(A,k(X)^{\times})$ est nulle. \qed
\end{theo}

La proposition \ref{Fischer}
implique aussi un \'enonc\'e sur la cohomologie non ramifi\'ee de degr\'e quelconque
(cf. \cite{ctsb}) :

\begin{theo} Soient  
  $i\geq 0$ et $m\geq  1$ des entiers, $j\in \Z$. 
 
(i) Tout \'el\'ement normalis\'e de $H^{i}_{nr}(k(X)^G,\mu_{m}^{\otimes j})$
a une image nulle dans $H^{i}(k(X)^A,\mu_{m}^{\otimes j})$  
pour tout sous-groupe ab\'elien $A \subset G$
d'ordre impair.

(ii) Si la condition $Cyc(G,k)$ est satisfaite,
  tout \'el\'ement normalis\'e de $H^{i}_{nr}(k(X)^G,\mu_{m}^{\otimes j})$
a une image nulle dans $H^{i}(k(X)^A,\mu_{m}^{\otimes j})$  
pour tout sous-groupe ab\'elien $A \subset G$.\qed
\end{theo}

   \medskip
 
{\bf Notation}  Soit $G$ un groupe fini. Pour tout $G$-module $M$ et tout entier $i \geq 0$,
 on d\'efinit les sous-groupes
$${\cyr{X}}^{i}_{ab}(G,M) \subset {\cyr{X}}^{i}_{bicyc}(G,M)   \subset {\cyr{X}}^{i}_{cyc}(G,M)  \subset H^{i}(G,M)$$
comme les sous-groupes form\'es des \'el\'ements de $H^{i}(G,M)$ dont la restriction
\`a chaque sous-groupe cyclique, resp. bicyclique, resp. ab\'elien $H \subset G$
est nulle. Par d\'efinition, un  groupe bicyclique est un groupe ab\'elien engendr\'e 
 par  deux \'el\'ements.
 
 Un $G$-module de permutation est un groupe ab\'elien libre qui admet une base
 respect\'ee par l'action de $G$. C'est   une somme directe de $G$-modules $\Z[G/H]$
 pour divers sous-groupes $H \subset G$. On a $H^1(G,\Z)=\Hom(G,\Z)=0$,
 et $${\cyr{X}}^{2}_{cyc}(G,\Z) \simeq {\cyr{X}}^{1}_{cyc}(G,\Q/\Z) =0,$$
 car tout caract\`ere de $G$ nul sur tout \'el\'ement de $G$ est nul.
 En utilisant le lemme de Shapiro, on \'etend ces r\'esultats sur le $G$-module  $\Z$
 aux $G$-modules de permutation. Ceci permet d'\'etablir le :

\begin{lem}\label{lemmesha} Soit $G$ un groupe fini.

(i) Pour $M$ un $G$-module de permutation, on a $H^1(G,M)=0$.

(ii) Pour $M$ un $G$-module sans torsion,   on a $$H^1(G,M\otimes \Q/\Z) \iso H^2(G,M)$$ et, si l'action de $G$ est triviale, $H^1(G,M)=0$.

(iii) Pour  $M$ un $G$-module avec action triviale,   les groupes
$${\cyr{X}}^{1}_{ab}(G,M) \subset {\cyr{X}}^{1}_{bicyc}(G,M)   \subset {\cyr{X}}^{1}_{cyc}(G,M)  $$
sont nuls.

(iv)  Pour  $M$ un $G$-module de permutation, les groupes
$${\cyr{X}}^{2}_{ab}(G,M) \subset {\cyr{X}}^{2}_{bicyc}(G,M)   \subset {\cyr{X}}^{2}_{cyc}(G,M)$$
sont nuls.

(v) Soit $$0 \to A \to B \to C \to 0$$
une suite exacte de $G$-modules. Si $C$ est un $G$-module de permutation, 
  cette suite
induit des isomorphismes ${\cyr{X}}^{2}_{ab}(A) \oi {\cyr{X}}^{2}_{ab}(B)$,
${\cyr{X}}^{2}_{bicyc}(A) \oi {\cyr{X}}^{2}_{bicyc}(B)$, ${\cyr{X}}^{2}_{cyc}(A) \oi {\cyr{X}}^{2}_{cyc}(B)$.

(vi) Soit $$0 \to A \to B \to C \to 0$$
une suite exacte de $G$-modules. Si $C$ est un $G$-module sans torsion
 avec action triviale de $G$, cette suite
induit des isomorphismes ${\cyr{X}}^{2}_{ab}(A) \oi {\cyr{X}}^{2}_{ab}(B)$,
${\cyr{X}}^{2}_{bicyc}(A) \oi {\cyr{X}}^{2}_{bicyc}(B)$, ${\cyr{X}}^{2}_{cyc}(A) \oi {\cyr{X}}^{2}_{cyc}(B)$.
\qed
\end{lem}

\begin{theo}\label{theoprincip2} 
Supposons la condition $Cyc(G,k)$ satisfaite. Soit $\mu(k)$ le groupe
des racines de l'unit\'e dans~$k$.

(a) La fl\`eche naturelle $$H^2(G,k^{\times}) \to H^2(G,k(X)^{\times})$$
induit un isomorphisme entre un sous-groupe de
$ {\cyr{X}}^2_{ab}(G,k^{\times})$
et le groupe  $H^2_{nr}(G,k(X)^{\times})$.

(b) La fl\`eche $H^2(G,\mu(k)) \to H^2(G,k^{\times})$ induit un isomorphisme  de groupes   
 $${\cyr{X}}^2_{ab}(G,\mu(k)) \oi  {\cyr{X}}^2_{ab}(G,k^{\times}),$$ et 
la fl\`eche compos\'ee $$H^2(G,\mu(k)) \to H^2(G,k^{\times}) \to H^2(G,k(X)^{\times})$$ identifie 
un sous-groupe du groupe   ${\cyr{X}}^2_{ab}(G,\mu(k))$
avec
le groupe $$\Br_{nr}^0(k(X)^G) \subset H^2(G,k(X)^{\times}).$$
 \end{theo}
  \begin{proof}  
Toute fonction inversible sur  $X=SL_{n,k}$ est constante, et   le groupe de
Picard  de $SL_{n,k}$ est nul.
La fl\`eche diviseur $ k(X)^{\times} \to \Div(X)$
sur le corps des fonctions de $X$ donne donc naissance \`a une suite exacte de $G$-modules
$$ 1 \to k^{\times} \to k(X)^{\times} \to   \Div(X)   \to 0.$$
Le $G$-module $\Div(X)$ est   le groupe ab\'elien libre sur
les points de codimension~1 de $X$. C'est donc un $G$-module de permutation.
  
  Le lemme \ref{lemmesha} (v)
  donne alors 
  $$\cyr{X}^2_{ab}(G,k^{\times}) \oi \cyr{X}^2_{ab}(G,k(X)^{\times}).$$
D'apr\`es le th\'eor\`eme \ref{theoprincip1},
on a une injection
$$H^2_{nr}(G,k(X)^{\times}) \hookrightarrow \cyr{X}^2_{ab}(G,k(X)^{\times}),$$
ce qui \'etablit (a).

Pour tout corps $k$ on a  la suite exacte de $G$-modules (avec action triviale)
 $$ 1 \to \mu(k) \to k^{\times} \to k^{\times}/\mu(k) \to 1$$
o\`u $R:=k^{\times}/\mu(k)$ est sans torsion.
D'apr\`es le 
  lemme \ref{lemmesha} (vi), on a 
   $${\cyr{X}}^2_{ab}(G,\mu(k)) \oi {\cyr{X}}^2_{ab}(G,k^{\times}).$$
  L'\'enonc\'e (b) r\'esulte alors de l'\'enonc\'e (a).
\end{proof}

   \begin{rema}
 Pour la torsion impaire, ou sous l'hypoth\`ese $Cyc(G,k)$, cet argument donne
 une autre d\'emonstration de la finitude de $\Br_{nr}^0(k(X)^G)$ (lemme \ref{lemmeinitial}).  Il suffit d'observer 
 que le groupe $H^2(G,\mu(k))$, et donc aussi 
 ${\cyr{X}}^2_{ab}(G,\mu(k)) \subset H^2(G,\mu(k))$, est fini. 
 Soit $l$ premier divisant l'ordre de $G$.
 Le groupe $\mu(k)\{l\}$ est soit fini soit isomorphe \`a $\Q_{l}/\Z_{l}$.
Dans les deux cas, le groupe $H^2(G, \mu(k)\{l\})$ est fini.
C'est clair dans le premier cas. Dans le second, $H^2(G, \Q_{l}/\Z_{l}) \iso H^3(G,\Z)\{l\}$.
  \end{rema}

\begin{cor}\label{racines}
Soit  $\mu(k)$ le groupe
des racines de l'unit\'e dans $k$. 

(a)  Si  le groupe  $\mu(k)$ est fini et  d'ordre premier 
 \`a celui de $G$ (ce qui implique que l'ordre de $G$ est  impair et   $Cyc(G,k)$ satisfaite), alors
$\Br_{nr}^0(k(X)^G)=0$. 

(b) Pour $k=\Q$ et $G$ d'ordre impair, on a 
$\Br(\Q)=\Br_{nr}(\Q(X)^G)$.

(c) Soit  $k=\R$.
  
(c1) On a   $\Br_{nr}^0(\R(X)^G) \subset {\cyr{X}}^2_{ab}(G,\Z/2)$.

(c2) On a $2.\Br_{nr}(\R(X)^G)=0$.

(c3) Si les 2-sous-groupes de Sylow de $G$ sont ab\'eliens, alors  $$\Br^0_{nr}(\R(X)^G)=0.$$
\end{cor}
\begin{proof}
 Dans le cas (a), on a $H^2(G,\mu(k))=0$, car ce groupe est annul\'e par l'ordre de $G$
 et par l'ordre de $\mu(k)$.
   Le th\'eor\`eme \ref{theoprincip2} donne alors~(a), et (b) est alors une cons\'equence imm\'ediate,
   puisque $\mu(\Q)=\Z/2$.
   Sur $\R$, la condition $Cyc(G,\R)$ est trivialement satisfaite, et $\mu(\R)=\Z/2$.
   Ceci donne (c1), et les \'enonc\'es (c2) et (c3) s'en suivent.
\end{proof}

\begin{rema}
  Dans la situation du (a) du corollaire \ref{racines}  ci-dessus, sur un corps de nombres, 
 par une m\'ethode arithm\'etique,
 Demarche \cite[Thm. 5]{demarche} a \'etabli la nullit\'e du 
 groupe $\Br_{nr,1}^0(k(X)^G) \subset \Br_{nr}^0(k(X)^G)$.
   \end{rema}

 \begin{rema}
Comme me l'a indiqu\'e C.~Demarche, les $p$-groupes qu'il \'etudie dans
\cite[\S 6]{demarche} fournissent, pour $p=2$, un exemple de groupe fini $G$
d'ordre~64
tel que le   groupe de Brauer  non ramifi\'e normalis\'e  $\Br^0_{nr} (\Q(X)^G)$ poss\`ede un \'el\'ement
alg\'ebrique  qui ne s'annule pas dans $\Br^0_{nr} (\R(X)^G)$ .
 \end{rema}

\begin{theo}\label{bogogeneral}
Supposons que pour tout groupe cyclique  $\Z/m \subset G$, on~a $\mu_{m}(k)=\mu_{m}(\k)$.
 On a alors des isomorphismes de groupes finis
 
 $${\cyr{X}}^2_{ab}(G,k^{\times})  \oi {\cyr{X}}^2_{bicyc}(G,k^{\times}) \oi H^2_{nr}(G,k(X)^{\times})$$
 
\noindent
 et 
  
  $${\cyr{X}}^2_{ab}(G,\mu(k)) \oi {\cyr{X}}^2_{bicyc}(G,\mu(k))  \oi H^2_{nr}(G,k(X)^{\times}).$$
  \end{theo}
  \begin{proof} L'hypoth\`ese  faite sur les racines de l'unit\'e 
implique que $Cyc(G,k)$ est satisfaite.

D'apr\`es le th\'eor\`eme  \ref{theoprincip1}, on a
 les inclusions
$$H^2_{nr}(G,k(X)^{\times}) \subset {\cyr{X}}^2_{ab}(G,k(X)^{\times}) 
 \subset   {\cyr{X}}^2_{bicyc}(G,k(X)^{\times}) .$$

L'hypoth\`ese  faite sur les racines de l'unit\'e  implique aussi que pour  tout anneau de valuation discr\`ete de rang~1 de
$k(X)^G$,  contenant $k$, 
 l'action des sous-groupes de d\'ecomposition de $G$  sur l'inertie
 est triviale. On peut donc copier l'argument de Bogomolov (cf. \cite[Thm. 6.1, Thm. 7.1]{cts}),
 qui m\`ene \`a l'inclusion
 $$ {\cyr{X}}^2_{bicyc}(G,k(X)^{\times}) \subset H^2_{nr}(G,k(X)^{\times}).$$
 Dans $H^2(G,k(X)^{\times}) $, on a donc les \'egalit\'es
 $$ {\cyr{X}}^2_{ab}(G,k(X)^{\times}) = {\cyr{X}}^2_{bicyc}(G,k(X)^{\times}) =H^2_{nr}(G,k(X)^{\times}).$$

Comme dans la d\'emonstration du  th\'eor\`eme \ref{theoprincip2},
on d\'eduit  du 
 lemme \ref{lemmesha} 
 les isomorphismes
  $$ {\cyr{X}}^2_{ab}(G, k^{\times}) \oi {\cyr{X}}^2_{ab}(G,k(X)^{\times}) $$
 $$ {\cyr{X}}^2_{bicyc}(G,k^{\times})  \oi {\cyr{X}}^2_{bicyc}(G,k(X)^{\times}).$$
 $$ {\cyr{X}}^2_{bicyc}(G,\mu(k))  \oi  {\cyr{X}}^2_{bicyc}(G,k^{\times})$$
$${\cyr{X}}^2_{ab}(G,\mu(k))  \oi  {\cyr{X}}^2_{ab}(G,k^{\times}).$$
\end{proof}

 Notons $\hat{G}= \Hom(G,\Q/\Z)=\Hom(G^{ab},\Q/\Z)$ le groupe des caract\`eres du groupe  $G$.

\begin{theo}\label{racinesfinies}
Supposons que le corps $k$ ne contient qu'un nombre fini de
racines de l'unit\'e, et  
supposons  la condition $Cyc(G,k)$ satisfaite.

(i) Soit $r>0$ tel que   $\mu(k)=\mu_{r}(k)$.
Le  groupe de Brauer non ramifi\'e alg\'ebrique normalis\'e 
$$Ker [\Br_{nr}^0(k(X)^G) \to \Br({\overline k}(X)^G)]= Ker[H^2_{nr}(G,k(X)^{\times}) \to H^2(G,\k(X)^{\times})]$$ s'identifie \`a un sous-groupe de $Ker [\hat{G}/r \to \prod_{A}  \hat{A}/r],$ o\`u $A$ parcourt les sous-groupes ab\'eliens de $G$.

(ii) Si $G$ est un groupe simple non ab\'elien, on a $\Br_{nr}^0(k(X)^G)=0$.
\end{theo}
\begin{proof} Le th\'eor\`eme \ref{theoprincip2} donne un plongement
$$ \Br_{nr}^0(k(X)^G)  \hookrightarrow  {\cyr{X}}^2_{ab}(G,\mu(k))  \subset H^2(G,\mu(k))$$
et donc aussi un plongement
$$ \Br_{nr}^0(k(X)^G)  \hookrightarrow {\cyr{X}}^2_{ab}(G,\mu_{r}(k)).$$

Par ailleurs sur $\overline k$, d'apr\`es Bogomolov (Thm. \ref{bogogeneral} ci-dessus)
on a un isomorphisme
$$ \Br_{nr}^0({\overline k}(X)^G)  \oi {\cyr{X}}^2_{ab}(G ,\mu(\k)) \subset H^2(G,\mu(\k)).$$

Le fl\`eches $\Br_{nr}^0(k(X)^G)  \to   H^2(G,\mu(k) )$  et 
$ \Br_{nr}^0({\overline k}(X)^G)  \to H^2(G,\mu(\k))$ sont compatibles, comme on le v\'erifie imm\'ediatement.
L'\'enonc\'e (i)  r\'esulte alors de la suite exacte de $G$-modules \`a action triviale
$$1 \to \mu_{r}(k) \to \Q/\Z(1) \to \Q/\Z(1) \to 1$$
d\'efinie par la multiplication par $r$ sur $\mu(\k)=\Q/\Z(1)$.

Si $G$ est comme en (ii), alors d'une part $\hat{G}=0$, d'autre part
$$\Br_{nr}^0({\overline k}(X)^G) =0,$$ d'apr\`es Kunyavski\u{\i} \cite{kuny}.
L'\'enonc\'e r\'esulte (ii) r\'esulte donc de (i).
\end{proof}

\begin{rema}
Le th\'eor\`eme  \ref{racinesfinies} s'applique aux   corps de nombres et  plus g\'en\'e\-ralement aux corps de type fini sur un corps de nombres, mais aussi aux corps $p$-adiques et aussi aux r\'eels.  Sur de tels corps, on voit donc
que si le groupe $G$ n'a  pas de caract\`eres, 
le groupe de Brauer non ramifi\'e alg\'ebrique  normalis\'e est nul. 
Sur un corps de nombres, 
ceci avait \'et\'e \'etabli par Harari
\cite[Prop. 4]{harari}, qui donne une formule pour le groupe de Brauer non ramifi\'e alg\'ebrique normalis\'e
comme sous-groupe de $H^1(\frak{g}, \Hom(G,\Q/\Z(1)))$, formule qui 
implique clairement que ce groupe est nul
 si $G$ n'a pas de caract\`eres.
\end{rema}

Pour les r\'eels, on obtient  l'\'enonc\'e  suivant.

\begin{cor}
Soit $k=\R$.   
Le groupe de Brauer non ramifi\'e alg\'ebrique normalis\'e  de $\R(X)^G$
$$Ker [\Br_{nr}^0(\R(X)^G) \to \Br(\C(X)^G)]$$ s'identifie \`a un sous-groupe de
$Ker [ \hat{G}/2  \to  \prod_{A \hskip1mm} \hat{A}/2],$
o\`u $A$ parcourt les sous-groupes ab\'eliens de $G$. \qed
\end{cor}

\end{document}